\documentclass [11pt,oneside]{amsart}

\usepackage{amssymb} \usepackage{amsfonts} \usepackage{amsmath}
\usepackage{amsthm} \usepackage{epsfig} 
\usepackage{tikz}
\usepackage[cmtip,arrow]{xy}
\usepackage{amsmath,amssymb,enumerate,pb-diagram,pb-xy}

\usepackage[applemac]{inputenc}
\input xy
\xyoption{all}

\usepackage{caption}

\newtheorem{lemma}{Lemma}[section]
\newtheorem{teo}[lemma]{Theorem}

\theoremstyle{definition}

\newtheorem{quest}[lemma]{Question}

\newtheorem{rem}[lemma]{Remark}

\newcommand{\matB}{\mathbb{B}}

\newcommand{\matN}{\ensuremath {\mathbb{N}}}

\newcommand{\calH} {\ensuremath {\mathcal{H}}}
\newcommand{\R} {\ensuremath {\mathbb{R}}}

\newcommand{\calC} {\ensuremath {\mathcal{C}}}

\author{Roberto Frigerio}

\address{Dipartimento di Matematica \\
Universit\`a di Pisa \\
Largo B.~Pontecorvo 5 \\
56127 Pisa, Italy}

\email{frigerio@dm.unipi.it}

\title[A note on  measure homology]{A note on  measure homology}



\thanks{}

\begin{document}
\begin{abstract}
Measure homology was introduced by Thurston in order to compute the simplicial volume
of hyperbolic manifolds. Berlanga endowed measure homology  
with a structure of graded locally convex (possibly non-Hausdorff) topological vector space.
In this note we completely characterize Berlanga's topology on measure homology
of CW-complexes,
showing in particular
that it is Hausdorff.
This answers a question posed by Berlanga.
\end{abstract}
\maketitle

\section{Introduction}
Measure homology was introduced by Thurston in~\cite{Thurston},
where it was exploited to compute the 
simplicial volume of closed hyperbolic manifolds.
It is proved in~\cite{Hansen,Zastrow}
that measure homology is canonically isomorphic to the usual real singular homology, 
at least for CW-complexes. Moreover, Loeh proved in~\cite{Loh} that
these homology theories
are not only isomorphic, but also isometric (with respect to
Thurston's seminorm on measure homology~\cite{Thurston}
and Gromov's seminorm on singular homology~\cite{Gromov}), a fact that plays a fundamental
r\^ole in applications to the simplicial volume.
For a comprehensive account about the notion
of measure homology and its applications see \emph{e.g.}~the introductions
of \cite{Zastrow, Berlanga}.

Thurston's seminorm is not the only extra-structure naturally
supported by measure homology.
In~\cite{Berlanga}, for every $n\in\matN$ the $n$-th measure homology 
module $\calH_n(X)$ 
of a topological space $X$ is endowed with a natural structure of locally convex (possibly non-Hausdorff)
topological vector space (this structure is not related to the topology induced by Thurston's
seminorm, see Remark~\ref{last} at the end of the paper). 

Let us now recall the main result of~\cite{Berlanga}:

\begin{teo}\label{berlanga}
Suppose that $X$ has the homotopy type of a countable CW-complex. Then
$\calH_1(X)$ is a locally convex Hausdorff vector space.
\end{teo}

In~\cite{Berlanga} Berlanga asks the following:

\begin{quest}\label{main:quest}
 Is $\calH_n(X)$ a Hausdorff space in general?
\end{quest}

In this note 
we completely characterize Berlanga's topology on measure homology, 
giving in particular a positive answer to Question~\ref{main:quest}, at least in the case of CW-complexes.
If $W$ is a real vector space, then the \emph{strongest weak topology} on $W$ is the weakest topology
which makes every linear functional on $W$ continuous (see Subsection~\ref{weak:sub} below). Our main result is the following:

\begin{teo}\label{main:teo}
Suppose that $X$ has the homotopy
type of a  CW-complex and let $n\in\matN$. Then Berlanga's topology
on $\calH_n(X)$ coincides with the strongest weak topology. In particular, it
is Hausdorff. 
\end{teo}

\section{Preliminaries}
\subsection{Measure homology}
Let $X$ be a topological space and let
 $S_n (X)$ be the set of singular 
$n$-simplices with values in $X$. We
endow $S_n(X)$ with the compact-open
topology and denote by $\matB_n(X)$ the $\sigma$-algebra
of Borel subsets of $S_n(X)$. 
If $\mu$ is a signed measure on $\matB_n(X)$ (in this case we say for short
that $\mu$ is a Borel measure on $S_n(X)$), then $\mu$ has finite total variation
if $|\mu(A)|<\infty$ for every $A\in\matB_n(X)$.
For every $n\geq 0$, the measure chain module
$\mathcal{C}_n(X)$
is the real vector space of the Borel measures on $S_n(X)$
having finite total variation and admitting a compact determination set
(see~\cite{Zastrow} for the definition of determination set and a detailed discussion
of the relationship between this notion and the notion of support).
The graded module $\mathcal{C}_\ast (X)$ can be given the structure
of a complex via the boundary operator
$$
\partial_{n}\colon \calC_{n}(X)\to \calC_{n-1}(X)\, ,\qquad 
\partial_n\mu =\sum_{j=0}^{n}(-1)^j\mu^j\, , 
$$
where $\mu^j$
is the push-forward of $\mu$ 
under the map that takes a simplex $\sigma\in S_n(X)$ into the composition of 
$\sigma$ with the usual inclusion
of the standard $(n-1)$-simplex onto the $j$-th face of $\sigma$.
The homology of the complex $(\calC_\ast(X),\partial_\ast)$ is the \emph{measure homology of $X$}, 
and it is denoted by $\mathcal{H}_\ast(X)$.

\subsection{Berlanga's topology}
If $\lambda\colon S_n(X)\to\R$ is a continuous function, 
then 
the map
$$
\psi_\lambda\colon \calC_n(X)\to\R\, ,\quad \mu\mapsto \int_{S_n(X)} \lambda\, {\rm d}\mu
$$
is a well-defined linear functional on $\calC_n(X)$.
Following~\cite{Berlanga}, we put on $\calC_n(X)$ the
weakest topology which makes $\psi_\lambda$ continuous for every continuous
$\lambda\colon S_n(X)\to\R$. We also put on $\calH_n(X)$ the quotient topology
induced by the restriction of the topology of $\calC_n(X)$ to the subspace 
$\ker \partial_n\subseteq \calC_n(X)$ of measure cycles.

\subsection{Topological vector spaces}\label{weak:sub}
By a topological vector space we mean a real vector space endowed
with a topology $\tau$ such that the vector space operations are continuous
with respect to $\tau$ (in particular, we don't require that $\tau$ is Hausdorff). 
It is readily seen that Berlanga's topology endows $\calC_n(X)$
with a structure of locally convex topological vector space. Since local convexity
is inherited by subspaces and quotients,
the vector space $\calH_n(X)$ is a locally convex (possibly non-Hausdorff)
topological vector space.

Let us now put Berlanga's definition into the general context of weak topologies associated
to pairings.
If $V,W$ are real vector spaces, a \emph{pairing} between $V$ and $W$ is simply a bilinear
map $\eta\colon V\times W\to\R$. Let $W^*$ be the algebraic dual of $W$. We say that
an element $\beta\in W^*$ is \emph{represented} by $v\in V$ 
via $\eta$ if $\beta=\eta(v,\cdot)$, and we denote by $W^*(\eta)\subseteq W^*$
the subset of functionals that are represented via $\eta$ by some element of $V$.
The \emph{weak topology}
on $W$ corresponding to $\eta$
is the weakest topology which makes every element of $W^*(\eta)$ continuous.
This topology endows $W$ with a structure of 
locally convex topological vector space, and it is Hausdorff if and only
if for every $w\in W$ there exists $v\in V$ such that $\eta(v,w)\neq 0$ (see \emph{e.g.}~\cite[Section 16]{KN}).

We define the \emph{strongest weak topology} $\tau^W_{\rm sw}$ on $W$ as the weak topology
associated to any pairing $\eta$ such that $W^*(\eta)=W^*$. 
So $\tau^W_{\rm sw}$ endows $W$ with a structure of
locally convex Hausdorff vector space.

\begin{rem}\label{weak:rem}
Let $\tau_{\rm slc}^W$ be the strongest locally convex topology on $W$
 (see \emph{e.g.}~\cite{KN} for the definition and several properties of $\tau_{\rm slc}^W$).
If $W$ is finite-dimensional, then $\tau_{\rm sw}^W=\tau_{\rm slc}^W$, and they both coincide
with the Euclidean topology on $W$.
Otherwise, $\tau_{\rm sw}^W$ is strictly weaker than  $\tau_{\rm slc}^W$. In fact,
let $p\colon W\to \R$ be a norm. The unit ball of $p$ does not contain any
nontrivial linear subspace of $W$, while 
every $\tau^W_{\rm sw}$-neighbourhood of $0\in W$ contains a finite-codimensional linear subspace of $W$. Therefore, if $\dim W=\infty$ then $p$ is not continuous with respect
to $\tau_{\rm sw}^W$. On the other hand, every norm is continuous
with respect to $\tau_{\rm slc}^W$, so   
$\tau^W_{\rm sw}\neq \tau^W_{\rm slc}$ unless the dimension of $W$ is finite.
\end{rem}

\subsection{Singular homology vs.~measure homology}
Let $X$ be a topological space.
We denote by $C_*(X)$ the complex of real singular chains 
of $X$, and by
$C^*(X)$
the complex of real singular cochains of $X$, 
\emph{i.e.}~the algebraic dual complex of $C_*(X)$.
We also let $H_*(X)$ (resp.~$H^*(X)$) be the homology of the complex
$C_*(X)$ (resp.~$C^*(X)$), \emph{i.e.}~the usual real singular homology (resp.~cohomology) of $X$.
For every $\sigma\in S_n(X)$, $n\in\matN$, let us denote by $\delta_\sigma\in\calC_n(X)$ the atomic
measure supported by the singleton $\{\sigma\}$. The chain
map
$$
\iota_\ast\colon C_\ast(X)\to \calC_\ast(X)\, ,\qquad
\iota_*\left(\sum_{i=0}^ka_i\sigma_i\right)=\sum_{i=0}^k a_i\delta_{\sigma_i}\, ,
$$
induces a map
$$
H_*(\iota_\ast)\colon H_*(X)\longrightarrow \calH_*(X)\ .
$$
The following result was proved independently by Zastrow and Hansen:

\begin{teo}[\cite{Hansen,Zastrow}]\label{ZH}
Suppose that $X$ has the homotopy type of a CW-complex. Then
the map
$$
H_n(\iota_n)\colon H_n(X)\longrightarrow \calH_n(X)
$$
is an isomorphism for every $n\in\matN$.
\end{teo}

\subsection{Continuous cohomology}
Let us now recall the definition of continuous cohomology of a topological space $X$. 
We
regard $S_n(X)$ as a subset of $C_n(X)$, so that for every cochain
$\varphi\in C^n(X)$ it makes sense to consider
the restriction $\varphi|_{S_n(X)}$. We say that
$\varphi$ is \emph{continuous} if $\varphi|_{S_n(X)}$ is, and we set
$$
C_c^\ast (X)=\{\varphi\in C^\ast(X)\, |\, \varphi\ {\rm is\ continuous} \}\ .
$$
It is readily seen that 
$C_c^\ast (X)$ is a subcomplex of $C^\ast (X)$. Its homology
is the \emph{continuous cohomology} of $X$, and it is denoted
by $H_c^\ast (X)$. The following result describes the relationship between
continuous cohomology and the usual singular cohomology.

\begin{teo}\label{fri}
Suppose that $X$ has the homotopy type of a CW-complex. Then
the inclusion 
$\rho^*\colon C_c^*(X)\hookrightarrow C^*(X)$ induces isomorphisms
$$
H^n(\rho^n)\colon H_c^n(X)\to H^n(X)\, , \quad n\in\matN\ .
$$
\end{teo}
\begin{proof}
By~\cite[Theorem 2]{Milnor}, any CW-complex has the homotopy type of a simplicial
complex, endowed with the \emph{metric} topology in the sense of Eilenberg and Steenrod~\cite[page 75]{ES}. Therefore, $X$ has the homotopy type of a locally contractible
metrizable space, and the conclusion follows from~\cite[Theorem 1.1]{Frigerio}.
\end{proof}

\section{Proof of Theorem~\ref{main:teo}}
Let $X$ be a topological space. For every $n\in\matN$, the map
$$
C_c^n(X)\times \calC_n(X)\to\R,\qquad (\varphi,\mu)\mapsto \int_{S_n(X)}\varphi\, {\rm d}\mu
$$
induces a pairing
$$
\eta_n\colon H_c^n(X)\times\calH_n(X)\to\R\ .
$$
It is very easy to compare this pairing with the usual Kronecker pairing
$$
\kappa_n\colon H^n(X)\times H_n(X)\to \R
$$
between singular homology and singular cohomology. In fact,
it readily follows from the definitions that
\begin{equation}\label{commuta}
\eta_n( \alpha,H_n(\iota_n)(\beta))=\kappa_n (H^n(\rho^n)(\alpha), \beta)  \quad
\textrm{for\ every}\ \alpha\in H_c^n(X),\, \beta\in H_n(X)\ .
\end{equation}

By construction,
Berlanga's topology on $\calH_n(X)$
coincides with the weak topology associated to the pairing $\eta_n$. 
Therefore, 
Theorem~\ref{main:teo} may be restated as follows:

\begin{teo}\label{main2:teo}
Suppose that $X$ has the homotopy type of a CW-complex. 
Then 
every
linear functional on $\calH_n(X)$ is represented
via $\eta_n$ by some element in $H^n_c(X)$.
\end{teo}
\begin{proof}
Let $\beta\colon\calH_n(X)\to \R$ be a fixed linear functional. 
The Universal Coefficient Theorem ensures that 
the composition $\beta\circ H_n(\iota_n)\colon H_n(X)\to\R$ is represented 
via
$\kappa_n$ by an element
in $H^n(X)$, \emph{i.e.}~that there exists
$\gamma\in H^n(X)$ such that 
\begin{equation}\label{sec:eq}
\kappa_n(\gamma, c)=\beta(H_n(\iota_n)(c))\quad \textrm{for\ every}\
c\in H_n(X)\ .
\end{equation}
Recall now from Theorem~\ref{fri} that the map
$H^n(\rho^n)\colon H^n_c(X)\to H^n(X)$ is an isomorphism, and set
$\gamma_c=H^n(\rho^n)^{-1}(\gamma)\in H^n_c(X)$. 
From Equations~\eqref{commuta} and~\eqref{sec:eq}
we deduce that 
$$
\eta_n(\gamma_c, H_n(\iota_n)(c))=
\kappa_n(\gamma, c)=
\beta(H_n(\iota_n)(c))\quad \textrm{for\ every}\ c\in H_n(X)\ .
$$
By Theorem~\ref{ZH}, the map $H_n(\iota_n)\colon H_n(X)\to \calH_n(X)$
is an isomorphism, so the last equation implies that 
$\eta_n(\gamma_c,\cdot)=\beta$ on the whole $\calH_n(X)$.
We have thus shown that $\beta$ is represented by $\gamma_c$
via $\eta_n$, and this concludes the proof.
\end{proof}

We conclude the paper with the following remark, that describes
the (lack of) relationship between Berlanga's
topology the topology induced by Thurston's seminorm on measure homology.

\begin{rem}\label{last}
Let $X$ be a CW-complex, and let us call \emph{Thurston's topology} 
the locally convex
topology on $\calH_n(X)$ associated to Thurston's seminorm. We show here that
there are examples where Thurston's topology is not finer than Berlanga's one,
and viceversa. 

Recall that the simplicial volume
of the closed orientable surface $\Sigma_2$ of genus two is positive~\cite{Gromov,Thurston}.
Therefore, 
if $X_\infty$ is the disjoint union of a countable family of copies of $\Sigma_2$,
then it is easily checked
that Gromov's seminorm
on $H_2(X_\infty)$ is in fact a norm. Since
singular homology and measure homology are isometrically isomorphic~\cite{Loh},
the same is true
for Thurston's seminorm on $\calH_2(X_\infty)$. Since $\dim H_2(X_\infty)=\infty$, the argument
described in Remark~\ref{weak:rem} shows that in this case
Berlanga's topology is not finer than Thurston's one. On the other hand, 
every normed infinitely dimensional vector space admits a non-continuous
linear functional, and this implies that 
Thurston's topology on $\calH_2(X_\infty)$ is not finer than Berlanga's one.

If $\dim \calH_n(X)<\infty$, then Berlanga's topology is the
strongest locally convex topology on $\calH_n(X)$, and is therefore finer than Thurston's
one. More precisely, in the finite dimensional case Berlanga's topology
is strictly finer than Thurston's one if and only if Thurston's seminorm
is not a norm (this is the case, for example, for $\calH_1 (S^1)$).
\end{rem}

\bibliographystyle{amsalpha}
\bibliography{biblio}

\end{document}